\documentclass[12pt, reqno]{amsart}
\usepackage{amssymb}
\usepackage{amsfonts}
\usepackage{amssymb,amsmath,amscd}
\usepackage[colorlinks=true]{hyperref}
\usepackage{mathrsfs}
\usepackage{amsthm}
\newtheorem{Theorem}{\indent Theorem}[section]

\newtheorem{Lemma}[Theorem]{\indent Lemma}
\newtheorem{Corollary}[Theorem]{\indent Corollary}

\theoremstyle{remark}
\newtheorem{Remark}{Remark}

\begin{document}
\centerline{
\bf On $k$-free numbers over Beatty sequences
}
\bigskip
\centerline{Wei Zhang}

\centerline{
School of Mathematics and Statistics,
Henan University,
Kaifeng  475004, Henan, China}
\centerline{zhangweimath@126.com}
\bigskip

\textbf{Abstract}
In this paper, we consider $k$-free numbers over Beatty sequences. New results are given. In particular, for a fixed irrational number $\alpha>1$ of finite type $\tau<\infty$,   any constant $\varepsilon>0$,  we can show that
\begin{align*}
\sum_{\substack{1\leq n\leq x\\ [\alpha n+\beta]\in\mathcal{Q}_{k}}}
1-
\frac{x}{ \zeta(k)} \ll x^{k/(2k-1)+\varepsilon}
+x^{1-1/(\tau+1)+\varepsilon},
\end{align*}
where $\mathcal{Q}_{k}$ is
the set of positive $k$-free integers and the implied constant depends only on $\alpha,$ $\varepsilon,$ $k$ and $\beta.$ This improves previous results. The main new  ingredient of our idea is an employing of the double exponential sums of the type
\[
\sum_{1\leq h\leq H}\sum_{\substack{1\leq n\leq x\\ n\in\mathcal{Q}_{k}}}e(\vartheta hn).
\]



\bigskip
\numberwithin{equation}{section}

\section{\bf{Introduction}}
In this paper, we are interested in  $k-$free integers over Beatty sequences.  The so-called Beatty sequence of integers are defined by
$
\mathcal{B}_{\alpha,\beta}:=\{[\alpha n+\beta]\}_{n=1}^{\infty},
$
where $\alpha$ and $\beta$ are fixed real numbers and $[x]$ denotes the greatest integer not larger than $x$.
The analytic properties of such sequences have been studied by many experts. For example, one can refer to \cite{ABS,BS1,BY} and the references therein.
A number $q$ is called
$k$-free integer  if and only if
$
m^{k}|q\Longrightarrow m=1.
$
For sufficiently large $x\geq1,$ it
is well known that
\[
\sum_{n\in\mathcal{Q}_{k}}n^{-s}=
\frac{\zeta(s)}{\zeta(ks)},\ \ \ \Re s>1
\]
and
\begin{align}\label{kf}
\sum_{n\leq x,\ n\in\mathcal{Q}_{k}}1=
\frac{x}{\zeta(k)}+O(x^{1/k}),
\end{align}
where $\mathcal{Q}_{k}$ is
the set of positive $k$-free integers.
In this paper, we are interested in the sum
\[
\sum_{\substack{1\leq n\leq x\\ [\alpha n+\beta]\in\mathcal{Q}_{k}}}
1.\]
In fact, this problem has been considered by many experts. For example, in 2008, G\"{u}lo\u{g}lu and  Nevans  \cite{GC} proved that
\[
\sum_{\substack{1\leq n\leq x\\ [\alpha n+\beta]\in\mathcal{Q}_{2}}}
1=\frac{x}{ \zeta(2)} +O\left(\frac{x\log \log x}{\log x}\right),\]
where $\alpha>1$ is the
 irrational number of finite type.

%

%
%
Now we will recall some notion related to the type of $\alpha.$ The definition of an irrational number of constant type can be cited as follows.
For an irrational number $\alpha,$ we define its type $\tau$ by the relation
\[
\tau:=\sup\left\{\theta\in\mathbb{R}:\
\liminf_{\substack{q\rightarrow \infty\\ q\in\mathbb{Z}^{+}}}q^{\theta}\parallel
\alpha q\parallel=0\right\}.
\]
Let $\psi$ be a non-decreasing positive function that defined for integers. The irrational number $\alpha$ is said to be of type $<\psi$ if $q\parallel q\alpha\parallel\geq 1/\psi(q)$ holds for every positive integers $q.$ If $\psi$ is a constant function, then an irrational $\alpha$ is also called a constant type (finite type). This relation between these two definitions is that an individual number $\alpha$ is of type $\tau$ if and only if for every constant $\tau$, there is a constant $c(\tau,\alpha)$ such that $\alpha$ is of type $\tau$ with
$q\parallel q\alpha\parallel\geq c(\tau,\alpha)q^{-\tau-\varepsilon+1}.$

Recently, in \cite{Go,Di}, it is proved  that
\[
\sum_{\substack{1\leq n\leq x\\ [\alpha n+\beta]\in\mathcal{Q}_{2}}}
1=\frac{x}{ \zeta(2)} +O\left(Ax^{5/6}(\log x)^{5}\right),\]
where
$
A=\max\{\tau(m), 1\leq m \leq x^{2}\},
$
and $\alpha>1$ is  fixed irrational algebraic number.
More recently, Kim, Srichan and Mavecha \cite{KSM} improved the above result by showing that
\[
\sum_{\substack{1\leq n\leq x\\ [\alpha n+\beta]\in\mathcal{Q}_{k}}}
1=\frac{x}{ \zeta(k)} +O\left(x^{(k+1)/2k}(\log x)^{3}\right).\]


Recently, with some much more generalized arithmetic functions, in \cite{ABS,TZ}, one may also get some other estimates for such type sums. However, the estimates of \cite{ABS,TZ} cannot be applied to an individual $\alpha.$
In this paper, we can give the following formula.
\begin{Theorem}\label{th2}
Let $\alpha>1$ be a fixed irrational number of finite type $\tau<\infty$. Then for any constant $\varepsilon>0$,  we have
\begin{align*}
\sum_{\substack{1\leq n\leq x\\ [\alpha n+\beta]\in\mathcal{Q}_{k}}}
1-
\alpha^{-1}\sum_{\substack{1\leq n\leq [\alpha x+\beta]\\ n\in\mathcal{Q}_{k}}}1
\ll   x^{k/(2k-1)+\varepsilon}+x^{1-1/(\tau+1)+\varepsilon},
\end{align*}
where the implied constant depends only on $\alpha,$ $\varepsilon,$ $k$ and $\beta.$
\end{Theorem}

Then by  (\ref{kf}) and the above theorem, we can obtain  the follows.
\begin{Corollary}
Let $\alpha>1$ be a fixed irrational number of finite type $\tau<\infty$. Then for any constant $\varepsilon>0$,  we have
\begin{align*}
\sum_{\substack{1\leq n\leq x\\ [\alpha n+\beta]\in\mathcal{Q}_{k}}}
1-
\frac{x}{ \zeta(k)} \ll  x^{k/(2k-1)+\varepsilon}+x^{1-1/(\tau+1)+\varepsilon},
\end{align*}
where the implied constant depends only on $\alpha,$ $\varepsilon,$ $k$ and $\beta.$
\end{Corollary}
In fact, our result relies heavily on the following double sum.
\begin{Theorem}\label{IK}
Suppose  for some positive integers $a,q,h,$   $q\leq x,$ $h\leq H\ll x,$ $(a,q)=1$ and
\begin{align}\label{DAT}
\left|\vartheta-\frac{a}{q}\right|\leq \frac{1}{q^{2}},
\end{align}
then   for sufficiently large $x$ and any $\varepsilon>0,$ we have
\[
\sum_{1\leq h\leq H}\sum_{\substack{1\leq n\leq x\\ n\in\mathcal{Q}_{k}}}e(\vartheta hn)\ll \left(Hx^{k/(2k-1)}+ q+ Hx/q\right)x^{\varepsilon},
\]
where the implied  constant may depend on $k$ and $\varepsilon.$
\end{Theorem}
\begin{Remark}
One can also compare this result with the results of  Br\"{u}dern-Perelli \cite{BP} and Tolev  \cite{To}. By using the argument of \cite{BP,To}, one may get some better results for some special cases.
\end{Remark}
\section{Proof of Theorem \ref{th2}}
We will start  the proof by introducing some necessary lemmas.
%
%
%
%

\begin{Lemma}\label{z2}
Let $\alpha>1$ be of finite type $\tau<\infty$ and let $K$ be  sufficiently  large. For an integer $w\geq1,$ there exists $a,q\in\mathbb{N},$  $a/q\in\mathbb{Q}$ with $(a,q)=1$ and $q$ satisfying $K^{1/\tau-\varepsilon}w^{-1}<q\leq K$ such that
\[
\left|\alpha w-\frac{a}{q}\right|\leq \frac{1}{qK}.
\]
\end{Lemma}
\begin{proof}
By Dirichlet approximation theorem, there is a rational number $a/q$ with $(a,q)=1$ and $q\leq K$ such that
$
\left|\alpha w-a/q\right|<1/qK.
$
Then we have
$
\parallel qw\alpha\parallel\leq 1/K.
$
Since $\alpha$ is of type $\tau<\infty,$ for sufficiently large $K,$ we have
$
\parallel qw\alpha\parallel\geq (qw)^{-\tau-\varepsilon}.
$
Then we have
$
1/K\geq \parallel qw\alpha\parallel\geq (qw)^{-\tau-\varepsilon}.
$
This gives that
$
q\geq K^{1/\tau-\varepsilon}w^{-1}.
$
\end{proof}

In order to prove the theorem, we need the definition of the discrepancy. Suppose that we are given a sequence $u_{m},$ $m=1,2,\cdots$, $M$ of points of $\mathbb{R}/\mathbb{Z}.$ Then the discrepancy $D(M)$ of the sequence is
\begin{align}\label{di}
D(M)=\sup_{\mathcal{I}\in[0,1)}
\left|\frac{\mathcal{V}(\mathcal{I},M)}{M}
-|\mathcal{I}|\right|,
\end{align}
where the supremum is taken over all subintervals $\mathcal{I}=(c, d)$ of the interval $[0, 1),$
$\mathcal{V} (\mathcal{I}, M)$ is the number of positive integers $m\leq M$ such that $a_{m}\in\mathcal{I}$, and
$|\mathcal{I}| = d-c$ is the length of $|\mathcal{I}|.$

 Without lose generality, let $D_{\alpha,\beta}(M)$ denote the
discrepancy of the sequence $\{\alpha m+\beta\},$ $m=1,2,\cdots$, $M$,
where $\{x\}=x-[x].$
The following lemma  is from \cite{BS1}.
\begin{Lemma}\label{lere}
Let $\alpha>1.$ An integer $m$ has the form $m=[\alpha n+\beta]$ for some integer  $n$ if and only if
\[
0<\{\alpha^{-1}(m-\beta+1)\}\leq \alpha^{-1}.
\]
The value of $n$ is determined uniquely by $m.$
\end{Lemma}

\begin{Lemma}[See Theorem 3.2 of Chapter 2 in  \cite{KN}]\label{let}
Let $\alpha$ be a fixed irrational number of type $\tau<\infty.$ Then, for all $\beta\in\mathbb{R},$ we have
\[
D_{\alpha,\beta}(M)\leq M^{-1/\tau+o(1)}
,\ \ (M\rightarrow\infty),
\]
where the function implied by $o(1)$ depends only on $\alpha.$
\end{Lemma}

\begin{Lemma}[see page 32 of \cite{Vi}]\label{levi}
For any $\Delta\in\mathbb{R}$ such that $0<\Delta<1/8$ and $\Delta\leq 1/2\min\{\gamma,1-\gamma\},$ there exists a periodic function $\Psi_{\Delta}(x)$ of period 1 satisfying the following properties:
\begin{itemize}
\item $0\leq$$\Psi_{\Delta}(x)$$\leq1$ for all $x\in\mathbb{R};$

\item $\Psi_{\Delta}(x)=\Psi(x)$ if $\Delta\leq x\leq\gamma-\Delta$ or $\gamma+\Delta\leq x\leq 1-\Delta;$

\item $\Psi_{\Delta}(x)$ can be represented as a Fourier series
\[
\Psi_{\Delta}(x)=\gamma+\sum_{j=1}^{\infty}
g_{j}e(jx)+h_{j}e(-jx),
\]
\end{itemize}
where
 \begin{align*}
 \Psi(x)=
 \begin{cases}
 1\ \ &\textup{if}\ 0<x\leq \gamma,\\
 0\ \ &\textup{if}\ \gamma<x\leq 1,
\end{cases}
 \end{align*}
 and the coefficients $g_{j}$ and $h_{j}$ satisfy the upper bound
\[
\max\{|g_{j}|, |h_{j}|\}\ll \min\{j^{-1},j^{-2}\Delta^{-1}\},\ \ (j\geq1).
\]
\end{Lemma}

Suppose that $\alpha>1.$ Then we have that $\alpha$ and $\gamma=\alpha^{-1}$ are of the same type. This means that $\tau(\alpha)=\tau(\gamma)$ (see page 133 in \cite{BY}). Let $\delta=\alpha^{-1}(1-\beta)$ and $M=[\alpha x+\beta].$ Then by Lemma \ref{lere}, we have
\begin{align}\label{8}
\begin{split}
\sum_{\substack{1\leq n\leq x\\ [\alpha n+\beta]\in \mathcal{Q}_{k}}}1&=\sum_{\substack{1\leq m\leq M \\ 0<\{\gamma m+\delta\}\leq \gamma\\ m\in \mathcal{Q}_{k}}}1+O(1)\\
&=\sum_{\substack{1\leq m\leq M\\ m\in \mathcal{Q}_{k}}}\Psi(\gamma m+\delta)+O(1),
\end{split}
\end{align}
where $\Psi(x)$ is the periodic function with period one for which
 \begin{align*}
 \Psi(x)=
 \begin{cases}
 1\ \ &\textup{if}\ 0<x\leq \gamma,\\
 0\ \ &\textup{if}\ \gamma<x\leq 1.
\end{cases}
 \end{align*}
By a classical result of Vinogradov (see Lemma \ref{levi}), it is known that for any $\Delta$ such that
$
0<\Delta<1/8\ \textup{and} \ \Delta\leq \min\{\gamma,1-\gamma\}/2,
$
there is a real-valued function $\Psi_{\Delta}(x)$ satisfy the conditions of Lemma \ref{levi}. Hence, by (\ref{8}), we can obtain that
\begin{align}\label{7}
\begin{split}
\sum_{\substack{1\leq n\leq x\\ [\alpha n+\beta]\in \mathcal{Q}_{k}}}1
&=\sum_{\substack{1\leq m\leq M\\ m\in \mathcal{Q}_{k}}}\Psi(\gamma m+\delta)+O(1)\\
&=\sum_{\substack{1\leq m\leq M\\ m\in \mathcal{Q}_{k}}}\Psi_{\Delta}(\gamma m+\delta)+O(1+V(I,M)M^{\varepsilon}),
\end{split}
\end{align}
where $V(I,M)$ denotes the number of positive integers $m\leq M$ such that
\[
\{\gamma m+\delta\}\in I=[0,\Delta)\cup
(\gamma-\Delta,\gamma+\Delta)
\cup(1-\Delta,1).
\]
Since $|I|\ll \Delta,$ it follows from the definition (\ref{di}) and Lemma \ref{let} that
\begin{align}\label{5}
V(I,M)\ll \Delta x+x^{1-1/\tau+\varepsilon},
\end{align}
where the implied constant depends only on $\alpha.$
 By Fourier expansion for $\Psi_{\Delta}(\gamma m+\delta)$ (Lemma \ref{levi}) and changing the order of summation, we have
\begin{align}\label{6}
\begin{split}
\sum_{\substack{1\leq m\leq M\\m\in \mathcal{Q}_{k}}}& \Psi_{\Delta}(\gamma m+\delta)\\
&=\gamma\sum_{\substack{m\leq M\\ m\in \mathcal{Q}_{k}}}1+\sum_{k=1}^{\infty}g_{k}e(\delta k)\sum_{\substack{1\leq m\leq M\\ m\in \mathcal{Q}_{k}}}e(\gamma km)\\
&+\sum_{k=1}^{\infty}h_{k}e(-\delta k)\sum_{\substack{1\leq m\leq M\\ m\in \mathcal{Q}_{k}}}e(\gamma km).
\end{split}
\end{align}
By Theorem \ref{IK}, Lemma \ref{z2} and Lemma \ref{levi}, we see
that for $0<k \ll x^{(4k-4)/(2k-1)+\varepsilon}$,
we have
\begin{align}\label{1}
\begin{split}
\sum_{1\leq k\leq  x^{(4k-4)/(2k-1)+\varepsilon}}g_{k}e(\delta k)\sum_{\substack{1\leq m\leq M\\ m\in \mathcal{Q}_{k}}}e(\gamma km)
\ll  x^{k/(2k-1)+\varepsilon}
+x^{1-1/(\tau+1)+\varepsilon},
\end{split}
\end{align}
where we have also used the fact that
$\alpha$ and $\alpha^{-1}$ are of the same type (finite type).
Similarly, we have
\begin{align}\label{2}
\begin{split}
\sum_{1\leq k\leq  x^{(4k-4)/(2k-1)+\varepsilon}}h_{k}e(-\delta k)\sum_{\substack{1\leq m\leq M\\ m\in \mathcal{Q}_{k}}}e(\gamma km)
\ll x^{k/(2k-1)+\varepsilon}
+x^{1-1/(\tau+1)+\varepsilon},
\end{split}
\end{align}
On the other hand, the trivial bound
\[
\sum_{\substack{1\leq m\leq M\\ m\in \mathcal{Q}_{k}}}e(\gamma km)\ll x
\]
implies that
\begin{align}\label{3}
\begin{split}
\sum_{k\geq x^{(4k-4)/(2k-1)+\varepsilon}}g_{k}e(\delta k)\sum_{\substack{1\leq m\leq M\\ m\in\mathcal{Q}_{k}}}e(\gamma km)&\ll
x^{1+\varepsilon}\sum_{k\geq x^{(4k-4)/(2k-1)+\varepsilon}}k^{-2}\Delta^{-1}\\&\ll x^{k/(2k-1)+\varepsilon}
\end{split}
\end{align}
and
\begin{align}\label{4}
\begin{split}
\sum_{k\geq x^{(4k-4)/(2k-1)+\varepsilon}}h_{k}e(-\delta k)\sum_{\substack{1\leq m\leq M\\ m\in\mathcal{Q}_{k}}}e(\gamma km)&\ll
x^{1+\varepsilon}\sum_{k\geq x^{(4k-4)/(2k-1)+\varepsilon}}k^{-2}\Delta^{-1}\\&\ll x^{k/(2k-1)+\varepsilon}
\end{split}
\end{align}
where $\Delta=x^{-(k-1)/(2k-1)+\varepsilon} .$
Inserting the bounds (\ref{1})-(\ref{4}) into (\ref{6}), we have
\begin{align*}
\sum_{\substack{1\leq n\leq x\\ [\alpha n+\beta]\in\mathcal{Q}_{k}}}
1-\alpha^{-1}
\sum_{\substack{1\leq n\leq [\alpha x+\beta]\\ n\in\mathcal{Q}_{k}}}1
\ll   x^{k/(2k-1)+\varepsilon}+x^{1-1/(\tau+1)+\varepsilon},
\end{align*}
where the implied constant depends on $\alpha,$ $k,$ $\beta$ and $\varepsilon.$ Substituting this bounds and (\ref{5}) into (\ref{7}) and choosing $\Delta=x^{-(k-1)/(2k-1)+\varepsilon}$, we complete the proof of   Theorem \ref{th2}.

\section{Proof of Theorem \ref{IK}}

By Dirichlet hyperbolic method, we have
\begin{align*}
\sum_{1\leq h\leq H}\sum_{\substack{1\leq n\leq x\\ n\in \mathcal{Q}_{k}}}e(\alpha hn)&=
\sum_{1\leq h\leq H}\sum_{1\leq m^{k}\leq y}\mu(m)\sum_{1\leq l\leq x/m^{k}}  e(\alpha m^{k}lh)
\\&+
\sum_{1\leq h\leq H}\sum_{1\leq l\leq x/y} \sum_{1\leq m^{k}\leq x/l}\mu(m)e(\alpha m^{k}lh)\\
&-\sum_{1\leq h\leq H}\sum_{1\leq m^{k}\leq y}\mu(m)\sum_{1\leq l\leq x/y}  e(\alpha m^{k}lh),
\end{align*}
where $y$ is certain parameter to be chosen later. By the well known estimate
\[
\sum_{1\leq n\leq x}e(n\alpha)\leq \min\left(x,\frac{1}{2||\alpha||}\right),
\]
 we have
\begin{align*}
\sum_{1\leq h\leq H}\sum_{1\leq m^{k}\leq y}\mu(m)\sum_{1\leq l\leq x/m^{k}} e(\alpha m^{k}lh)
&\ll
\sum_{1\leq h\leq H}\sum_{1\leq m^{k}\leq y}
\min\left(x/m^{k},\frac{1}{2||\alpha hm^{k}||}\right)\\
&\ll
\sum_{1\leq h\leq H}\sum_{1\leq m \leq y}
\min\left(x/m ,\frac{1}{2||\alpha hm ||}\right)\\
&\ll
 (Hy)^\varepsilon\sum_{1\leq n \leq Hy}
\min\left(Hx/n ,\frac{1}{2||\alpha n ||}\right)\\
\\
&\ll
 (Hy)^\varepsilon \left(Hy +Hx/q+q\right),\\
\end{align*}
where we have used the following lemma.
\begin{Lemma}[See section 13.5 in \cite{IK}]\label{MS}
For
\[
\left|\theta-a/q\right|\leq q^{-2},
\]
$a,q\in\mathbb{N}$ and  $(a,q)=1,$ then we have
\[
\sum_{1\leq n \leq M}\min\left\{\frac{x}{n},\frac{1}{2\parallel n\theta\parallel}\right\}\ll \left(M+q+xq^{-1}\right)\log 2qx,
\]
where $\parallel u\theta\parallel$ denotes the distance of $u$ from the nearest integer.
\end{Lemma}

For the second the sum, we can use the exponential sum for M\"{o}bius function. We have
\[
\sum_{1\leq h\leq H}\sum_{1\leq l\leq x/y} \sum_{1\leq m^{k}\leq x/l}\mu(m)e(\alpha m^{k}lh)\ll Hx/y^{1-1/k}(\log x)^{-A},
\]
where $A$ is any positive constant.
For the third sum we need the following lemma.
\begin{Lemma}[See section 13.5 in \cite{IK}]\label{MS0}
For
\[
\left|\theta-a/q\right|\leq q^{-2},
\]
$a,q\in\mathbb{N}$ and  $(a,q)=1,$ then we have
\[
\sum_{1\leq n \leq M}\min\left\{x,\frac{1}{2\parallel n\theta\parallel}\right\}\ll \left(M+x+Mx/q+q\right)\log 2qx,
\]
where $\parallel u\theta\parallel$ denotes the distance of $u$ from the nearest integer.
\end{Lemma}

By Lemma \ref{MS0}, we have
\begin{align*}
\sum_{1\leq h\leq H}\sum_{1\leq m^{k}\leq y}\mu(m)\sum_{1\leq l\leq x/y} e(\alpha m^{k}lh)
&\ll
\sum_{1\leq h\leq H}\sum_{1\leq m^{k}\leq y}
\min\left(x/y,\frac{1}{2||\alpha hm^{k}||}\right)\\
&\ll
\sum_{1\leq h\leq H}\sum_{1\leq m \leq y}
\min\left(x/y ,\frac{1}{2||\alpha hm ||}\right)\\
&\ll
 (Hy)^\varepsilon\sum_{1\leq n \leq Hy}
\min\left(x/y ,\frac{1}{2||\alpha n ||}\right)\\
\\
&\ll
 (Hy)^\varepsilon \left(Hy +x/y+Hx/q+q\right).\\
\end{align*}

Choosing $y=x^{k/(2k-1)}$ completes the proof of Theorem \ref{IK}.

\bigskip
$\mathbf{Acknowledgements}$
I am deeply grateful to the referee(s) for carefully reading the manuscript and making useful suggestions.

\address{Wei Zhang\\ School of Mathematics\\
               Henan University\\
               Kaifeng  475004, Henan, China}
\email{zhangweimath@126.com}
\end{document}